\pdfoutput=1 
\documentclass[10pt,letterpaper,leqno,twoside]{amsart}
\usepackage[utf8]{inputenc}
\usepackage[colorlinks=true, pdfstartview=FitV, linkcolor=blue, citecolor=blue, urlcolor=blue]{hyperref}
\usepackage[english]{babel}
\usepackage{amsmath}
\usepackage{amsfonts}
\usepackage{amssymb}
\usepackage{amsthm}
\usepackage{graphicx}
\usepackage[protrusion=true,expansion=true]{microtype}
\usepackage[left=3cm,right=3cm,top=2.7cm,bottom=2.7cm]{geometry}
\usepackage[babel=true]{csquotes}

\author{Dimitri Dias}
\title[Integral Generalized Apollonian Sphere Packings]{The Local-Global Principle for Integral Generalized Apollonian Sphere Packings}

\newtheorem{theorem}{Theorem}

\newtheorem{lemma}{Lemma}

\newtheorem{definition}{Definition}

\newcommand{\e}{{\mathrm e}}
\newcounter{numrem}

\newenvironment{remark}[1][Remark \arabic{numrem} :]{\refstepcounter{numrem} \begin{trivlist}
\item[\hskip \labelsep {\bfseries #1}]}{\end{trivlist}}

\usepackage{lipsum}
\newenvironment{example}{\begin{quote}%
  \textbf{Example.}%
  \quad
}{%
\end{quote}%
}

\begin{document}

\begin{abstract}
Four mutually tangent spheres form two gaps. In each of these, one can inscribe in a unique way four mutually tangent spheres such that each one of these spheres is tangent to exactly three of the original spheres. Repeating the process gives rise to a generalized Apollonian sphere packing. These packings have remarkable properties. One of them is the local to global principle and will be proven in this paper.
\end{abstract}

\maketitle

\section{Introduction}

\begin{theorem} \label{gap}
Four mutually tangent spheres form two gaps. In each of these gaps, there is a unique way to inscribe four mutually tangent spheres in such a way that, for each one of these spheres, there is exactly one of the original spheres that is not tangent to this one (and the non-tangent sphere is different for each of the four inscribed spheres).
\end{theorem}

\begin{example} \label{octuref}
Consider the planes $z=1$ and $z=-1$ and the $2$ spheres of radius $1$ and centers $(-1,-1,0)$ and $(-1,1,0)$. These form a set of $4$ mutually tangent spheres, which define two gaps, in each of which we can inscribe $4$ spheres according to the procedure defined before.
\\

In one of the gaps, the two spheres of radius $1$ and centers $(1,1,0)$ and $(1,-1,0)$, and the two spheres of radius $\frac{1}{2}$ and centers $(0,0,\frac{1}{2})$ and $(0,0,-\frac{1}{2})$ can be inscribed.

\begin{center}
\includegraphics[scale=0.24]{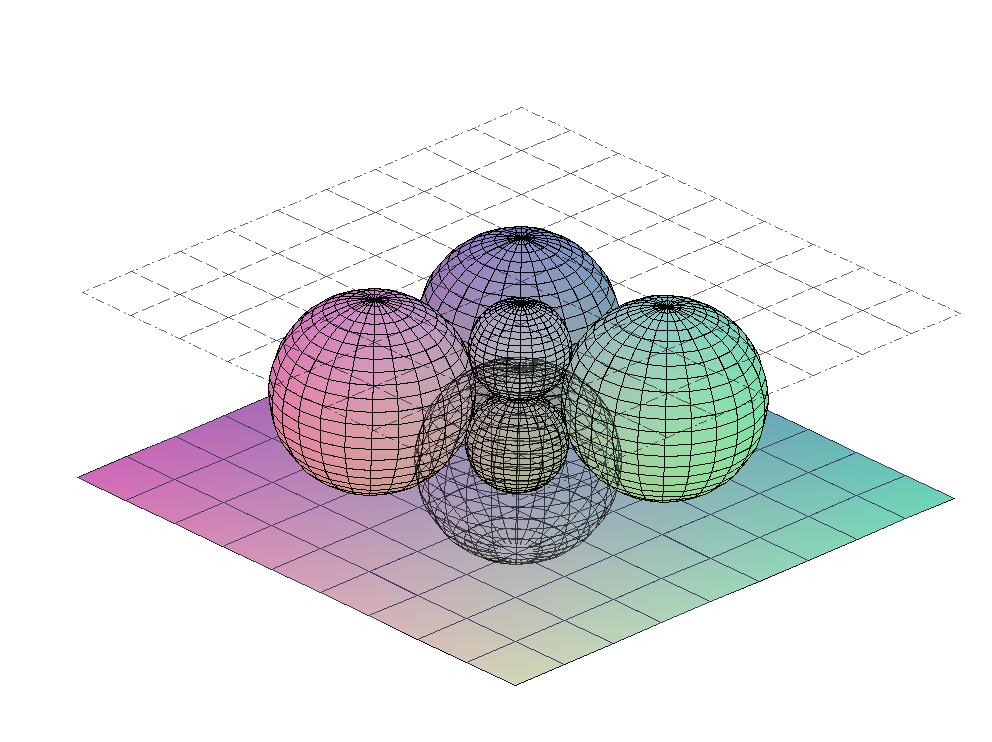} 
\end{center}
\end{example}

\begin{proof}[Proof of Theorem \ref{gap}]
Using a M\"{o}bius transformation, a rotation, a translation and a scaling, one can transform any configuration of four mutually tangent spheres into the configuration seen in Example \ref{octuref}. In this case, the result can be proven using simple geometry.
\end{proof}

Thus, in either of the two gaps formed by four mutually tangent spheres, we can pack four more spheres, forming an octuple. In each octuple, the spheres come in four pairs: each sphere can be paired with the only sphere in the octuple which does not touch it. The spheres of a pair do not touch each other, but are tangent to all the six other spheres in the octuple.
\\

Now, pick four mutually tangent spheres amongst those in the octuple. They again form two gaps, in which, using Theorem \ref{gap}, we can pack four new spheres. Repeating the process results in what we will call a generalized Apollonian sphere packing.
\\

This construction is a generalization to the $3$-dimensional case of the work of Guettler and Mallows in \cite{Mallows}. Analogously to the classical Apollonian circle or sphere packings, and to the generalized Apollonian circle packings, this kind of constructions is the source of several questions. One of them, which will be tackled here, concerns the curvatures (the inverses of the radii) of the spheres appearing in the packing.
\\

The first section will be dedicated to generalities. It will be shown that the curvatures can be described using the orbit under the action of a certain group of integral matrices on a certain vector depending on the packing. This will allow us to study the case of integral sphere packings, in which the spheres all have integer curvatures.
\\

In the second section, we will use a method due to Sarnak \cite{Sarnak} to prove the local-global principle for generalized Apollonian sphere packings. This method was originally used in an attempt to prove the strong positive density conjecture for Apollonian circle packings. It happened to be insufficient, and the proof needed further work by Bourgain and Fuchs \cite{Bourgain} and then Bourgain and Kontorovich \cite{Konto}. In the case of sphere packings, and similarly to the work of Kontorovich on the classical Apollonian sphere packings \cite{Konto2}, Sarnak's method happens to be strong enough, as being in a $3$-dimensional case adds more variables to the problem which ultimately comes to studying the integers represented by a quadratic form.
\\

It should be noted that a similar, but independent, proof has been very recently submitted on arXiv by K. Nakamura \cite{Nakam}.

\subsection*{Acknowledgements}
I would like to thank my supervisor, Professor Andrew Granville, for discussions that greatly helped me with my work and Professor Alex Kontorovich who took time to answer my questions. I would also like to thank my friends Oleksiy Klurman, Crystel Bujold, Kevin Henriot, Mohammad Bardestani, Daniel Fiorilli and Marzieh Mehdizadeh for their helpful comments.

\section{Generalities}

We begin with some generalities about generalized Apollonian sphere packings. Most of this section is a direct generalization of \cite{Mallows}. Their notations and method will be used here.
\\

A sphere $\mathcal{S}$ of curvature (or bend) $b$ (the inverse of its radius) and center $(x,y,z)$ can be described by its ``augmented bend, bend*center'' ($abbc$) coordinates $\textbf{a}(\mathcal{S})=(\overline{b},b,bx,by,bz)$, where $\overline{b}$ is the curvature of the sphere that is the inverse of $\mathcal{S}$ in the unit sphere, that is
\begin{equation*}
\overline{b}=b(x^2 + y^2 + z^2) - \frac{1}{b} \, .
\end{equation*}

For planes, which are considered as spheres of infinite radius, this definition needs to be modified. If our plane has equation $p_1x + p_2y + p_3z = h$, with $(p_1,p_2,p_3)$ a unit vector, we define its $abbc$ coordinates to be $(2h,0,p_1,p_2,p_3)$. 

\subsection{Octuples of spheres}\

In the following, an octuple configuration will denote a set of eight spheres obtained from the procedure of Theorem \ref{gap}. As explained in the introduction, an octuple can be seen as a set of four pairs of spheres, where a pair of spheres consists in a sphere in the octuple and the only sphere that is not tangent to it. The spheres in a pair will usually be denoted by $\mathcal{S}$ and $\mathcal{S}'$.

\begin{remark} \label{outer}
In the case where we have an outer sphere enclosing the seven other spheres, this sphere will be given a negative curvature. There can only be one sphere with a negative curvature in an octuple (and in a packing).
\end{remark}
Let
\begin{equation*}
\textbf{W} = \begin{pmatrix}
0 & -\frac{1}{2} & 0 & 0 & 0 \\ 
-\frac{1}{2} & 0 & 0 & 0 & 0 \\ 
0 & 0 & 1 & 0 & 0 \\ 
0 & 0 & 0 & 1 & 0 \\ 
0 & 0 & 0 & 0 & 1
\end{pmatrix} \, .
\end{equation*}

\begin{lemma} \label{lem1}
Let $\mathcal{S}_1$ and $\mathcal{S}_2$ be two spheres. Then,
\begin{align*}
\textbf{a}(\mathcal{S}_1) \textbf{W} \textbf{a}(\mathcal{S}_2)^t &= 1 \text { if } \mathcal{S}_1=\mathcal{S}_2 \\
&=-1 \text{ if } \mathcal{S}_1 \text{ and } \mathcal{S}_2 \text{ are externally tangent}.
\end{align*}
\end{lemma}

\begin{proof}
Easy computations.
\end{proof}

\begin{remark}
If, as in Remark \ref{outer}, we define the curvature of the outer sphere to be negative, we still have $\textbf{a}(\mathcal{S}_1) \textbf{W} \textbf{a}(\mathcal{S}_2)^t = -1$ if the spheres are internally tangent.
\end{remark}

\begin{lemma} \label{lem2}
Let $\mathcal{S}$ and $\mathcal{S}'$ be two non-tangent spheres in an octuple configuration. Then,
\begin{equation*}
\textbf{a}(\mathcal{S}) \textbf{W} \textbf{a}(\mathcal{S}')^t = -3 \, .
\end{equation*}
\end{lemma}

\begin{proof}
This can be easily proven for the octuple introduced in the previous example. Any octuple can be seen as the image by a M\"{o}bius transformation of this configuration, i.e., a composition of rotations, translations, scalings and inversions in the unit sphere. Thus, we just have to study the effect of rotations, translations, scalings and inversions in the unit sphere on the $abbc$ coordinates of a sphere.
\\
\\
Let $\textbf{a}(\mathcal{S})$ be the $abbc$ coordinates of a sphere $\mathcal{S}$. Then,
\begin{itemize}
\renewcommand{\labelitemi}{$\bullet$}
\item a scaling by $\lambda$ consists in replacing $\textbf{a}(\mathcal{S})$ by $\textbf{a}(\mathcal{S}) \textbf{m}$, where
\begin{equation*}
\textbf{m}=\begin{pmatrix}
\frac{1}{\lambda} & 0 & 0 & 0 & 0 \\ 
0 & \lambda & 0 & 0 & 0 \\ 
0 & 0 & 1 & 0 & 0 \\ 
0 & 0 & 0 & 1 & 0 \\ 
0 & 0 & 0 & 0 & 1
\end{pmatrix} \, ,
\end{equation*}
\item a rotation consists in replacing $\textbf{a}(\mathcal{S})$ by $\textbf{a}(\mathcal{S}) \textbf{m}$, where
\begin{equation*}
\textbf{m}=\begin{pmatrix}
1 & 0 & 0 & 0 & 0 \\ 
0 & 1 & 0 & 0 & 0 \\ 
0 & 0 & 1 & 0 & 0 \\ 
0 & 0 & 0 & \cos \theta & -\sin \theta \\ 
0 & 0 & 0 & \sin \theta & \cos \theta
\end{pmatrix} \text{ or }
\begin{pmatrix}
1 & 0 & 0 & 0 & 0 \\ 
0 & 1 & 0 & 0 & 0 \\ 
0 & 0 & \cos \theta & 0 & \sin \theta \\ 
0 & 0 & 0 & 1 & 0 \\ 
0 & 0 & -\sin \theta & 0 & \cos \theta
\end{pmatrix} \text{ or }
\begin{pmatrix}
1 & 0 & 0 & 0 & 0 \\ 
0 & 1 & 0 & 0 & 0 \\ 
0 & 0 & \cos \theta & -\sin \theta & 0 \\ 
0 & 0 & \sin \theta & \cos \theta &0 \\ 
0 & 0 & 0 & 0 & 1 
\end{pmatrix} \, ,
\end{equation*}
\item a translation by a vector $(x,y,z)$ consists in replacing $\textbf{a}(\mathcal{S})$ by $\textbf{a}(\mathcal{S}) \textbf{m}$, where
\begin{equation*}
\textbf{m}=\begin{pmatrix}
1 & 0 & 0 & 0 & 0 \\ 
x^2 + y^2 + z^2 & 1 & x & y & z \\ 
2x & 0 & 1 & 0 & 0 \\ 
2y & 0 & 0 & 1 &0 \\ 
2z & 0 & 0 & 0 & 1 
\end{pmatrix} \, ,
\end{equation*}
\item the inversion in the unit sphere consists in replacing $\textbf{a}(\mathcal{S})$ by $\textbf{a}(\mathcal{S}) \textbf{m}$, where
\begin{equation*}
\textbf{m}=\begin{pmatrix}
0 & 1 & 0 & 0 & 0 \\ 
1 & 0 & 0 & 0 & 0 \\ 
0 & 0 & 1 & 0 & 0 \\ 
0 & 0 & 0 & 1 &0 \\ 
0 & 0 & 0 & 0 & 1 
\end{pmatrix} \, .
\end{equation*}
\end{itemize}
Now, it can easily be checked that $\textbf{m} \textbf{W} \textbf{m}^t=\textbf{W}$ for any of the previous matrices.
\end{proof}

\begin{lemma}
In an octuple containing the spheres $\mathcal{S}_1$ and $\mathcal{S}_1'$, $\mathcal{S}_2$ and $\mathcal{S}_2'$, $\mathcal{S}_3$ and $\mathcal{S}_3'$, $\mathcal{S}_4$ and $\mathcal{S}_4'$, we have
\begin{equation*}
\textbf{a}(\mathcal{S}_1) + \textbf{a}(\mathcal{S}_1') = \textbf{a}(\mathcal{S}_2) + \textbf{a}(\mathcal{S}_2') = \textbf{a}(\mathcal{S}_3) + \textbf{a}(\mathcal{S}_3') = \textbf{a}(\mathcal{S}_4) + \textbf{a}(\mathcal{S}_4') \, .
\end{equation*}
\end{lemma}

\begin{proof}
For $1 \leq j \leq 4$, let $\textbf{w}_j=\frac{\textbf{a}(\mathcal{S}_j)+\textbf{a}(\mathcal{S}_j')}{2}$.
\\

From Lemmas \ref{lem1} and \ref{lem2}, we have that $\textbf{w}_i \textbf{W} \textbf{w}_j^t = -1$ for every $1 \leq i,j \leq 4$. Therefore, if $\textbf{F}_j$ is the $5 \times 5$ matrix with rows $\textbf{a}(\mathcal{S}_1), \textbf{a}(\mathcal{S}_2), \textbf{a}(\mathcal{S}_3), \textbf{a}(\mathcal{S}_4)$ and $\textbf{w}_j$, we have
\begin{equation} \label{FKF}
\textbf{F}_i \textbf{W} \textbf{F}_j^t = \textbf{K} = \textbf{F}_j \textbf{W} \textbf{F}_j^t
\end{equation}
with
\begin{equation*}
\textbf{K}=\begin{pmatrix}
1 & -1 & -1 & -1 & -1 \\ 
-1 & 1 & -1 & -1 & -1 \\ 
-1 & -1 & 1 & -1 & -1 \\ 
-1 & -1 & -1 & 1 & -1 \\ 
-1 & -1 & -1 & -1 & -1
\end{pmatrix} \, .
\end{equation*}
This gives us that $\textbf{F}_i=\textbf{F}_j$ for every $1 \leq i,j \leq 4$. This implies that $\textbf{w}_i = \textbf{w}_j$ for every $1 \leq i,j \leq 4$.
\end{proof}

This result allows us to use the following convenient representation of any octuple.

\begin{definition} \label{defi}
Given an octuple containing the spheres $\mathcal{S}_1$ and $\mathcal{S}_1'$, $\mathcal{S}_2$ and $\mathcal{S}_2'$, $\mathcal{S}_3$ and $\mathcal{S}_3'$, $\mathcal{S}_4$ and $\mathcal{S}_4'$, we define a matrix $\textbf{F}$ associated to the octuple to be a matrix whose first four rows are the $abbc$ coordinates of four of the spheres (one from each pair) and the fifth row is the average of the $abbc$ coordinates of the two spheres of any pair.
\end{definition}

\begin{remark}
Since we can choose the pairs in different orders and since we have two choices of a representative for each pair, there are (at most) $384$ different $\textbf{F}$ matrices associated to the same octuple.
\end{remark}

\begin{theorem} \label{eq}
Let $\mathcal{S}_1$, $\mathcal{S}_2$, $\mathcal{S}_3$ and $\mathcal{S}_4$ be four mutually tangent spheres with respective curvatures $b_1,b_2,b_3,b_4$, enclosing two gaps. Then, the curvatures of the two sets of four spheres that can be inscribed in these gaps are given by $(2\omega - b_1,2\omega - b_2,2\omega - b_3,2\omega - b_4)$ and $(2\omega' - b_1,2\omega' - b_2,2\omega' - b_3,2\omega' - b_4)$, where $\omega$ and $\omega'$ are the roots of
\begin{equation} \label{equa}
2\omega^2 - 2 \omega (b_1+b_2+b_3+b_4) + b_1^2+b_2^2+b_3^2+b_4^2=0
\end{equation}
\end{theorem}

\begin{proof}
Using (\ref{FKF}), we have that $\textbf{F} \textbf{W} \textbf{F}^t = \textbf{K}$, thus $\textbf{F}^t \textbf{K}^{-1} \textbf{F} = \textbf{W}^{-1}$. Looking at the $(2,2)$ element of this equation gives us (\ref{equa}).
\end{proof}

\subsection{Generalized sphere packings}\

Given four mutually tangent spheres, enclosing two gaps, we have a unique way to inscribe four mutually tangent spheres in each of these gaps using the construction of Theorem \ref{gap}. The four original mutually tangent spheres belong to two different octuples. Equation (\ref{equa}) of Theorem \ref{eq} allows us to pass from one to another, since it implies that $\omega + \omega' = b_1 + b_2 + b_3 + b_4$.
\\

More precisely, if $\textbf{F}$ is the matrix associated to one of these octuples, with its first four rows containing the $abbc$ coordinates of the four original spheres, then the other octuple can be described by the matrix $\textbf{A}_5 \cdot \textbf{F}$, where
\begin{equation*}
\textbf{A}_5=
\begin{pmatrix}
1 & 0 & 0 & 0 & 0 \\ 
0 & 1 & 0 & 0 & 0 \\ 
0 & 0 & 1 & 0 & 0 \\ 
0 & 0 & 0 & 1 & 0 \\ 
1 & 1 & 1 & 1 & -1
\end{pmatrix} \, .
\end{equation*}

Any octuple of spheres is described by a matrix $\textbf{F}$, which explicitly gives the $abbc$ coordinates of four of the spheres in the octuples. Using the definition of the last row of $\textbf{F}$, the $abbc$ coordinates of the other spheres can be retrieved by looking at the products $\textbf{A}_i \cdot \textbf{F}$, $1 \leq i \leq 4$, where
\begin{align*}
\textbf{A}_1 = \begin{pmatrix}
-1 & 0 & 0 & 0 & 2 \\ 
0 & 1 & 0 & 0 & 0 \\ 
0 & 0 & 1 & 0 & 0 \\ 
0 & 0 & 0 & 1 & 0 \\ 
0 & 0 & 0 & 0 & 1
\end{pmatrix},  \quad \textbf{A}_2 = \begin{pmatrix}
1 & 0 & 0 & 0 & 0 \\ 
0 & -1 & 0 & 0 & 2 \\ 
0 & 0 & 1 & 0 & 0 \\ 
0 & 0 & 0 & 1 & 0 \\ 
0 & 0 & 0 & 0 & 1
\end{pmatrix}, \\ \textbf{A}_3 = \begin{pmatrix}
1 & 0 & 0 & 0 & 0 \\ 
0 & 1 & 0 & 0 & 0 \\ 
0 & 0 & -1 & 0 & 2 \\ 
0 & 0 & 0 & 1 & 0 \\ 
0 & 0 & 0 & 0 & 1
\end{pmatrix}, \quad \textbf{A}_4 = \begin{pmatrix}
1 & 0 & 0 & 0 & 0 \\ 
0 & 1 & 0 & 0 & 0 \\ 
0 & 0 & 1 & 0 & 0 \\ 
0 & 0 & 0 & -1 & 2 \\ 
0 & 0 & 0 & 0 & 1
\end{pmatrix} \, .
\end{align*}

Repeating the process of inscribing spheres in the gaps results in a generalized Apollonian sphere packings. Therefore, the set of $abbc$ coordinates of spheres in the packing is exactly the set of the first four coordinates of matrices in the orbit $\mathcal{A} \cdot \textbf{F}$, where $\mathcal{A}$ is the group
\begin{equation*}
\mathcal{A} = \langle \textbf{A}_1,\textbf{A}_2,\textbf{A}_3,\textbf{A}_4,\textbf{A}_5 \rangle \, .
\end{equation*}

\begin{remark}
The packing can be constructed from any octuple it contains.
\end{remark}

\begin{lemma} \label{inte}
If the curvatures of any octuple in a generalized Apollonian sphere packing are integers (or, equivalently, the second column of a matrix $\textbf{F}$ associated to this octuple has integer coordinates), then the same holds for all the spheres in the packing. Such a packing is called an integral generalized Apollonian sphere packing.
\\

If the curvatures of any octuple in an integral generalized Apollonian sphere packing are coprime (or, equivalently, the second column of a matrix $\textbf{F}$ associated to this octuple has coprime coordinates), then the same holds for any octuple in the packing. Such a packing is called a primitive integral generalized Apollonian sphere packing.
\end{lemma}

\begin{proof}
Let $(b_0,b_1,b_2,b_3,b_4,b'_0,b'_1,b'_2,b'_3,b'_4)$ be any integral octuple in the packing and $\omega$ the average of the curvatures in each pair. We have $2\omega \in \mathbb{Z}$. Now, from (\ref{eq}), we also have that $2 \omega^2 \in \mathbb{Z}$. Therefore, $\omega \in \mathbb{Z}$. Reciprocally, if $\omega \in \mathbb{Z}$ and $b_0,b_1,b_2,b_3,b_4$ are all integers, the same holds for all the spheres in the octuple.
\\
\\
Hence, the integrality of an octuple of curvatures is the same as the integrality of the second column of the matrix associated to this octuple.
\\

The group $\mathcal{A}$ is generated by matrices with integer coefficients. Therefore, if the second column of a matrix $\textbf{F}$ associated to some octuple in the packing has integer coordinates, the same will be true for the second column of the matrix associated to any octuple in the packing, since we just have to look at matrices in the orbit $\mathcal{A} \cdot \textbf{F}$.
\\
\\
Suppose that the octuple $(b_0,b_1,b_2,b_3,b_4,b'_0,b'_1,b'_2,b'_3,b'_4)$ is such that
\begin{equation*}
gcd(b_0,b_1,b_2,b_3,b_4,b'_0,b'_1,b'_2,b'_3,b'_4)=1 \, .
\end{equation*}
We have $gcd(b_1,b_2,b_3,b_4,b'_1,b'_2,b'_3,b'_4) =gcd(b_1,b_2,b_3,b_4,2\omega)=1$. Therefore, $gcd(b_1,b_2,b_3,b_4,\omega)=1$.
\\
\\
Reciprocally,
\begin{equation*}
gcd(b_1,b_2,b_3,b_4,\omega)=1 \Longrightarrow gcd(b_1,b_2,b_3,b_4,2\omega)=1 \text{ or } 2 \, .
\end{equation*}
But it cannot be $2$, since from (\ref{eq}) we would have $2| \omega$ and then $gcd(b_1,b_2,b_3,b_4,\omega)=2$. Hence, $gcd(b_0,b_1,b_2,b_3,b_4,b'_0,b'_1,b'_2,b'_3,b'_4)=gcd(b_1,b_2,b_3,b_4,2\omega)=1$.
\\

Now, one can easily see that the multiplication by any matrix of $\mathcal{A}$ does not change the $gcd$ of $(b_1,b_2,b_3,b_4,\omega)$.
\end{proof}

\begin{lemma} \label{parit}
Let $b_1,b_2,b_3,b_4$ be the curvatures of four mutually tangent spheres in a primitive octuple. Then, amongst $b_1,b_2,b_3,b_4$, there are two even numbers and two odd numbers. Furthermore, the two odd numbers are congruent modulo $4$.
\end{lemma}

\begin{proof}
First, we reduce equation (\ref{eq}) $\bmod{2}$. This yields
\begin{equation*}
b_1+b_2+b_3+b_4 = 0 \bmod{2} \, .
\end{equation*}
We have $3$ possibilities for the parities of the $b_1,b_2,b_3,b_4$
\begin{itemize}
\renewcommand{\labelitemi}{$\bullet$}
\item $b_1,b_2,b_3,b_4$ are all even,
\item $b_1,b_2,b_3,b_4$ are all odd,
\item there are $2$ even elements and $2$ odd elements amongst $b_1,b_2,b_3,b_4$.
\end{itemize}
$b_1,b_2,b_3,b_4$ cannot all be even by primitivity, since we would have $gcd(b_1,b_2,b_3,b_4,2\omega)=2$, which is impossible.
\\
\\
Suppose that $b_1,b_2,b_3,b_4$ are all odd. Looking at equation (\ref{eq}) $\bmod{4}$ gives $2\omega^2 = 0 \bmod{4}$, thus $\omega$ is even. Now, $\bmod{8}$, the same equation gives $b_1^2 + b_2^2 + b_3^2 +b_4^2 = 0 \bmod{8}$, but, since $b_1,b_2,b_3,b_4$ are all odd, $b_1^2 + b_2^2 + b_3^2 +b_4^2=4 \bmod{8}$.
\\
\\
Therefore, there are $2$ even elements and $2$ odd elements amongst $b_1,b_2,b_3,b_4$. Looking at the equation (\ref{eq}) $\bmod{4}$, we deduce that $\omega$ is odd. Suppose, without loss of generality, that $b_1,b_2$ are the two odd elements and that they are not congruent modulo $4$. Looking at equation (\ref{eq}) $\bmod{8}$ yields
\begin{equation*}
-2\omega(b_3+b_4)+(b_3+b_4)^2 \equiv 4 \bmod{8} \Rightarrow (b_3+b_4)(b_3+b_4-2\omega) \equiv 4 \bmod{8} 
\end{equation*}
But, since $\omega$ is odd, $b_3+b_4-2\omega \equiv b_3+b_4 -2 \bmod{4}$, which makes it impossible for $(b_3+b_4)(b_3+b_4-2\omega)$ to be $4 \bmod{8} $.
\end{proof}

\begin{lemma} \label{odd}
In any primitive integral packing, all the odd curvatures are congruent modulo $4$.
\end{lemma}

\begin{proof}
From the previous lemma, the two odd curvatures of our starting octuple are congruent modulo $4$. Now, one can easily check that, since $\omega$ is odd, none of the matrices of $\mathcal{A}$ changes the residue modulo $4$ of an odd element. 
\end{proof}

Similarly to the classical Apollonian circle or sphere packing case, or to the generalized Apollonian circle packing case, we can define the notion of root octuple, which is the \enquote{minimal} octuple in the packing, in the sense that it describes the biggest spheres in the packing (it is the octuple with the smallest $\omega$).
\begin{definition}
Let $v^t=(a,b,c,d,\omega)$ be an octuple in the packing $\mathcal{P}$. $v^t$ is said to be a root octuple if $a \leq 0 \leq b \leq c \leq d \leq \omega$ and $\omega \leq a+b+c+d$.
\end{definition}

\begin{remark}
Analogously to the classical Apollonian circle or sphere packing case, or to the generalized Apollonian circle packing case, the root octuple is unique. Like in these kinds of packings, there is a reduction algorithm that allows us, from any octuple in $\mathcal{P}$, to find the root octuple. 
\end{remark}

\section{The local-global principle}

In this section, we will use a method of Sarnark \cite{Sarnak} to prove the local-global principle for integral generalized Apollonian sphere packings. We are only interested in the curvatures and no longer need to use the full matrix $\textbf{F}$ associated to an octuple. We will only focus on the second column of this matrix, in which the first four coordinates are the curvatures of four spheres (one in each pair) in the octuple  and the fifth coordinate is the average of the curvatures in each pair.
\\

$\mathcal{P}$ will always denote a primitive integral generalized Apollonian sphere packing. Let $v_{\mathcal{P}}^t=(a_0,b_0,c_0,d_0,\omega_0)^t$ be the root octuple of the $\mathcal{P}$. Using Lemma \ref{parit}, two elements amongst $a_0,b_0,c_0,d_0$ are even and two are odd. Relabeling $a_0,b_0,c_0$ or $d_0$, we can assume that $a_0$ is even and $b_0$ is odd. If $a_0 = 0$, we can replace $a_0$ by $2\omega - a_0$ to make it non-zero. Notice that $a_0 + b_0$ will be odd and positive, since if $a_0 < 0$, we have $b_0 > \vert a_0 \vert$ (because in this case the sphere of curvature $b_0$ is inside the one of curvature $a_0$). We will again call this vector $v_{\mathcal{P}}^t$. Notice that such a $v_{\mathcal{P}}^t$ might no longer be a root quadruple.

\begin{theorem} \label{curva}
Let $v_{\mathcal{P}}^t=(a_0,b_0,c_0,d_0,\omega_0)^t$ as above. Then, the set of curvatures of spheres in $\mathcal{P}$ contains the set of integers of the form
\begin{equation*}
f_{a_0}(\alpha_1,\alpha_2,\beta_1,\beta_2) - a_0 \quad \text{with} \quad \text{gcd}_{\mathbb{Z}[i]}(\alpha_1 + i \alpha_2,\beta_1 + i\beta_2) = 1
\end{equation*}
where $f_{a_0}$ is the following positive definite integral quaternary quadratic form
\begin{equation*}
f_{a_0}(x,y,z,t)=A_0 x^2 + A_0 y^2 + 4D_0 z^2 + 4D_0 t^2 + 4xt B_0 - 4 yz B_0 + 4xz C_0 + 4yt C_0 
\end{equation*}
\begin{align*}
A_0 = a_0 + b_0, \quad
B_0 = -\frac{a_0 + b_0 + c_0 + d_0 - 2\omega_0}{2}, \\
C_0 = -\frac{a_0 + b_0 + c_0 - d_0}{2}, \quad
D_0 = a_0 + c_0 \, .
\end{align*}
\end{theorem}

\begin{proof}
As seen before, the set of curvatures of spheres in the packing is exactly the set of the first four coordinates of vectors in the orbit $\mathcal{A} \cdot v_{\mathcal{P}}^t$.
\\
\\
We will focus on the smaller orbit  $\mathcal{A}_1 \cdot v_{\mathcal{P}}^t$, where
\begin{equation*}
\mathcal{A}_1 = \langle \textbf{A}_2,\textbf{A}_3,\textbf{A}_4,\textbf{A}_5 \rangle \, .
\end{equation*}
This subgroup of $\mathcal{A}$ leaves the first coordinate of any vector invariant. Using Theorem \ref{eq}, we know that, for any vector $(a_0,b,c,d,\omega)$ in this orbit,
\begin{equation*}
2\omega^2 - 2 \omega (a_0+b+c+d) + a_0^2+b^2+c^2+d^2=0 \, .
\end{equation*}

The change of variables $(x_2,x_3,x_4,x_5) = (b + a_0, c + a_0, d + a_0,\omega + a_0)$ allows us to rewrite the equation as
\begin{equation} \label{eqmod}
Q(x_2,x_3,x_4,x_5)=2x_5^2 - 2 x_5(x_2 + x_3 + x_4) + x_2^2 + x_3^2 + x_4^2 = -2 a_0^2
\end{equation}
and, in the context of the orbit, is equivalent to conjugating the group $\mathcal{A}_1$ to $\mathcal{A}_1' = U^{-1} \mathcal{A}_1 U$  where
\begin{equation*}
U=\begin{pmatrix}
1 & 0 & 0 & 0 & 0 \\ 
-1 & 1 & 0 & 0 & 0 \\ 
-1 & 0 & 1 & 0 & 0 \\ 
-1 & 0 & 0 & 1 & 0 \\
-1 & 0 & 0 & 0 & 1 \\
\end{pmatrix} \, .
\end{equation*}
$\mathcal{A}_1'$ is isomorphic to $\Gamma = \langle \textbf{M}_2,\textbf{M}_3,\textbf{M}_4,\textbf{M}_5 \rangle$, where
\begin{align*}
\textbf{M}_2=\begin{pmatrix}
-1 & 0 & 0 & 2 \\ 
0 & 1 & 0 & 0 \\ 
0 & 0 & 1 & 0 \\
0 & 0 & 0 & 1
\end{pmatrix} \, , \quad&
\textbf{M}_3=\begin{pmatrix}
1 & 0 & 0 & 0 \\ 
0 & -1 & 0 & 2 \\ 
0 & 0 & 1 & 0 \\
0 & 0 & 0 & 1
\end{pmatrix} \, , \\
\textbf{M}_4=\begin{pmatrix}
1 & 0 & 0 & 0 \\ 
0 & 1 & 0 & 0 \\ 
0 & 0 & -1 & 2 \\
0 & 0 & 0 & 1
\end{pmatrix} \, , \quad&
\textbf{M}_5=\begin{pmatrix}
1 & 0 & 0 & 0 \\ 
0 & 1 & 0 & 0 \\ 
0 & 0 & 1 & 0 \\
1 & 1 & 1 & -1
\end{pmatrix} \, .
\end{align*}
The action of $\mathcal{A}_1$ on $v_0^t$ can be understood by studying the action of $\Gamma$ on $u_0^t=(x_2^0,x_3^0,x_4^0,x_5^0)^t = (b_0 + a_0,c_0 + a_0, d_0 + a_0,\omega_0 + a_0)^t$. We also have that $\Gamma \leq O_{Q}(\mathbb{Z})$.
\\
\\
We make another change of variables, $(x_2,x_3,x_4,x_5)=(A,D,A+2C+D,A+B+C+D)$. This allows us to rewrite the equation as
\begin{equation*}
\Delta(A,B,C,D)=B^2 + C^2 - AD = -a_0^2 \, .
\end{equation*}
In the context of the orbit, it is equivalent to conjugating the group $\Gamma$ to $G=V^{-1} \Gamma V$  where
\begin{equation*}
V=\begin{pmatrix}
1 & 0 & 0 & 0 \\ 
0 & 0 & 0 & 1 \\ 
1 & 0 & 2 & 1 \\ 
1 & 1 & 1 & 1
\end{pmatrix} \, .
\end{equation*}
Then, $G = \langle \textbf{g}_2,\textbf{g}_3,\textbf{g}_4,\textbf{g}_5 \rangle$ where
\begin{align*}
\textbf{g}_2 = \begin{pmatrix}
1 & 2 & 2 & 2 \\ 
0 & 0 & -1 & -1 \\ 
0 & -1 & 0 & -1 \\ 
0 & 0 & 0 & 1
\end{pmatrix} \, , \quad&
\textbf{g}_3 = \begin{pmatrix}
1 & 0 & 0 & 0 \\ 
-1 & 0 & -1 & 0 \\ 
-1 & -1 & 0 & 0 \\ 
2 & 2 & 2 & 1
\end{pmatrix} \, ,
\\
\textbf{g}_4 = \begin{pmatrix}
1 & 0 & 0 & 0 \\ 
0 & 0 & 1 & 0 \\ 
0 & 1 & 0 & 0 \\ 
0 & 0 & 0 & 1
\end{pmatrix} \, , \quad&
\textbf{g}_5 = \begin{pmatrix}
1 & 0 & 0 & 0 \\ 
0 & -1 & 0 & 0 \\ 
0 & 0 & 1 & 0 \\ 
0 & 0 & 0 & 1
\end{pmatrix} \, ,
\end{align*}
and $G \leq O_{\Delta}(\mathbb{Z})$. Let $G'= G \cap SO_{\Delta}(\mathbb{Z}) = \langle \textbf{g}_2\textbf{g}_3,\textbf{g}_2\textbf{g}_4,\textbf{g}_2\textbf{g}_5,\textbf{g}_3\textbf{g}_4,\textbf{g}_3\textbf{g}_5,\textbf{g}_4\textbf{g}_5 \rangle$.
\\

Proceeding as in Chapter 13.9 of \cite{Cassels} or as in \cite{Konto2}, we have the morphism
\begin{equation}
\begin{array}{ccc}\rho:\textrm{PSL}_2(\mathbb{C}) &\longrightarrow & \textrm{SO}_{\Delta}(\mathbb{R})\\
{\left(\begin{array}{ll}\alpha&\beta\\ \gamma&\delta\\ \end{array}\right)}&{\longmapsto}&{\begin{pmatrix}
|\alpha|^2 & 2 \Im(\beta \overline{\alpha}) & 2 \Re(\beta \overline{\alpha}) & |\beta|^2 \\ 
\Im(\alpha \overline{\gamma}) & \Re(\overline{\alpha} \delta - \overline{\beta} \gamma) & \Im(\alpha \overline{\delta} + \beta \overline{\gamma}) & \Im(\beta \overline{\delta}) \\ 
\Re(\alpha \overline{\gamma}) & \Im( \overline{\alpha} \delta - \overline{\beta} \gamma) & \Re(\alpha \overline{\delta} + \beta \overline{\gamma}) & \Re(\beta \overline{\delta}) \\ 
|\gamma|^2 & 2 \Im(\delta \overline{\gamma}) & 2 \Re(\delta \overline{\gamma}) & |\delta|^2
\end{pmatrix} 
}\\
\end{array} \, .
\end{equation} 

Let
\begin{align*}
\textbf{M}_1=\begin{pmatrix}
1 & 1+i \\ 
-1+i & -1
\end{pmatrix} \, , \quad&
\textbf{M}_2=\begin{pmatrix}
i & -1+i \\ 
0 & -i
\end{pmatrix} \, ,
\\
\textbf{M}_3=\begin{pmatrix}
(-1+i)\frac{\sqrt{2}}{2} & i\sqrt{2} \\ 
0 & (-1-i)\frac{\sqrt{2}}{2}
\end{pmatrix} \, , \quad&
\textbf{M}_4=\begin{pmatrix}
i & 0 \\ 
-1-i & -i
\end{pmatrix} \, ,
\\
\textbf{M}_5=\begin{pmatrix}
(-1+i)\frac{\sqrt{2}}{2} & 0 \\ 
-i\sqrt{2} & (-1-i)\frac{\sqrt{2}}{2}
\end{pmatrix} \, , \quad&
\textbf{M}_6=\begin{pmatrix}
(1+i)\frac{\sqrt{2}}{2} & 0 \\ 
0 & (1-i)\frac{\sqrt{2}}{2}
\end{pmatrix} \, .
\end{align*}
Then, we have
\begin{align*}
\rho(\textbf{M}_1)=\textbf{g}_2\textbf{g}_3 \, , \quad \rho(\textbf{M}_2)=\textbf{g}_2\textbf{g}_4 \, , \\
\rho(\textbf{M}_3)=\textbf{g}_2\textbf{g}_5 \, , \quad \rho(\textbf{M}_4)=\textbf{g}_3\textbf{g}_4 \, , \\
\rho(\textbf{M}_5)=\textbf{g}_3\textbf{g}_5 \, , \quad \rho(\textbf{M}_6)=\textbf{g}_4\textbf{g}_5 \, .
\end{align*}
Let $\mathcal{M}=\langle \textbf{M}_1, \textbf{M}_2, \textbf{M}_3, \textbf{M}_4, \textbf{M}_5, \textbf{M}_6 \rangle$. A brute force search gives
\begin{align*}
\textbf{M}_6 \textbf{M}_4 \textbf{M}_2 \textbf{M}_4 \textbf{M}_6^{-1} \textbf{M}_2^{-1} &=\begin{pmatrix}
1 & 2 \\ 
0 & 1
\end{pmatrix} \, , \quad
\textbf{M}_5^{-1} \textbf{M}_4^{-1} \textbf{M}_6^{-1} \textbf{M}_5^{-1} \textbf{M}_6^{-1} \textbf{M}_4 =\begin{pmatrix}
1 & 0 \\ 
2 & 1
\end{pmatrix} \, , \\
\textbf{M}_4^{-1} \textbf{M}_6^{-1} \textbf{M}_5 &=\begin{pmatrix}
1 & 0 \\ 
2i & 1
\end{pmatrix} \, , \quad
\textbf{M}_6^{-1} \textbf{M}_1 \textbf{M}_5 \textbf{M}_6 \textbf{M}_5^{-1} \textbf{M}_7 =\begin{pmatrix}
1 & 2i \\ 
0 & 1
\end{pmatrix} \, , \\
\textbf{M}_2 \textbf{M}_1 \textbf{M}_5 \textbf{M}_4^{-1} \textbf{M}_5 &=\begin{pmatrix}
i & 0 \\ 
0 & -i
\end{pmatrix} \, , \quad
\textbf{M}_3^{-1} \textbf{M}_2 \textbf{M}_6^{-1} \textbf{M}_4^{-1} \textbf{M}_3^{-1} \textbf{M}_5^{-1} \textbf{M}_4^{-1} \textbf{M}_6^{-1} \textbf{M}_5=\begin{pmatrix}
1+2i & 2 \\ 
2 & 1-2i
\end{pmatrix} \, , \\
\textbf{M}_4^{-1} \textbf{M}_5 \textbf{M}_6 &\textbf{M}_1^{-1} \textbf{M}_3^{-1} \textbf{M}_4^{-1} \textbf{M}_6^{-1} \textbf{M}_1^{-1} \textbf{M}_2^{-1} \textbf{M}_4=\begin{pmatrix}
1-2i & 2i \\ 
-2i & 1+2i
\end{pmatrix} \, , \\
\textbf{M}_6^{-1} \textbf{M}_2^{-1} &\textbf{M}_6^{-1} \textbf{M}_4^{-1} \textbf{M}_5^{-1} \textbf{M}_6^{-1} \textbf{M}_2^{-1}=\begin{pmatrix}
1+2i & 2i \\ 
-2i & 1-2i
\end{pmatrix} \, .
\end{align*}
This set of matrices generates the following subgroup of the Picard group $PSL_2(\mathbb{Z}[i])$
\begin{equation*}
\Xi = \Gamma(2) \cup \begin{pmatrix}
i & 0 \\ 
0 & -i
\end{pmatrix} \Gamma(2)
\end{equation*}
where $\Gamma(2)$ is the principal congruence subgroup of $PSL_2(\mathbb{Z}[i])$
\begin{equation*}
\Gamma(2) = \left\lbrace \begin{pmatrix}
\alpha & \beta \\ 
\gamma & \delta
\end{pmatrix} \in PSL_2(\mathbb{Z}[i]) \text{ such that } \begin{pmatrix}
\alpha & \beta \\ 
\gamma & \delta
\end{pmatrix} \equiv \begin{pmatrix}
 1 & 0 \\ 
0 & 1
\end{pmatrix} \bmod{2} \right\rbrace \, .
\end{equation*}
More precisely, as proven in \cite{Fine}, $\Gamma(2)$ is the normal closure of $\langle \begin{pmatrix}
1 & 2 \\ 0 & 1
\end{pmatrix}, \begin{pmatrix}
1 & 2i \\ 0 & 1
\end{pmatrix} \rangle$ in $PSL_2(\mathbb{Z}[i])$. Since the generators of $PSL_2(\mathbb{Z}[i])$ are known, one can compute all the conjugates and show that the matrices above generate a normal subgroup of $PSL_2(\mathbb{Z}[i])$, which therefore contains $\Gamma(2)$, hence contains $\Xi$.
\\
\\
Therefore,
\begin{align*}
V \rho(\Xi) V^{-1} \subset \Gamma \Rightarrow V \rho(\Xi ) V^{-1} . u_0^t \subset \Gamma . u_0^t \, .
\end{align*}
Let
\begin{equation*}
\begin{pmatrix}
A_0 \\ B_0 \\ C_0 \\ D_0
\end{pmatrix} = V^{-1} . u_0^t = \begin{pmatrix}
a_0 + b_0 \\ -\frac{a_0 + b_0 + c_0 + d_0 - 2\omega_0}{2} \\ -\frac{a_0 + b_0 + c_0 - d_0}{2} \\ a_0 + c_0
\end{pmatrix} \, .
\end{equation*}
From Lemma \ref{parit}, we know that this vector has integer coordinates. Using the definition of $\rho$, we have that the set of vectors of the form
\begin{equation} \label{set}
V \begin{pmatrix}
|\alpha|^2 & 2 \Im(\beta \overline{\alpha}) & 2 \Re(\beta \overline{\alpha}) & |\beta|^2 \\ 
\Im(\alpha \overline{\gamma}) & \Re(\overline{\alpha} \delta - \overline{\beta} \gamma) & \Im(\alpha \overline{\delta} + \beta \overline{\gamma}) & \Im(\beta \overline{\delta}) \\ 
\Re(\alpha \overline{\gamma}) & \Im( \overline{\alpha} \delta - \overline{\beta} \gamma) & \Re(\alpha \overline{\delta} + \beta \overline{\gamma}) & \Re(\beta \overline{\delta}) \\ 
|\gamma|^2 & 2 \Im(\delta \overline{\gamma}) & 2 \Re(\delta \overline{\gamma}) & |\delta|^2
\end{pmatrix} \begin{pmatrix}
A_0 \\ B_0 \\ C_0 \\ D_0
\end{pmatrix}
\end{equation}
with
\begin{equation*}
\begin{pmatrix}
\alpha & \beta \\ 
\gamma & \delta
\end{pmatrix} \in \Xi 
\end{equation*}
forms an explicit subset of the orbit $\Gamma . u_0^t$.
\\
\\
Considering the change of variables we used, the integers of the form $x_2 - a_0$, $x_3 - a_0$, $x_4 - a_0$ and $x_5 - a_0$, when $(x_2,x_3,x_4,x_5)^t$ runs through the orbit $\Gamma . u_0^t$, appear in the second, third, fourth and fifth coordinates of vectors in the orbit $\mathcal{A}_1 . v_0^t$. 
\\
\\
In particular, from (\ref{set}), we have that the integers of the form
\begin{equation*}
|\alpha|^2 A_0 + 2\Im(\beta \overline{\alpha})B_0 + 2\Re(\beta \overline{\alpha})C_0 + |\beta|^2 D_0 - a_0
\end{equation*}
with
\begin{equation*}
\begin{pmatrix}
\alpha & \beta \\ 
\gamma & \delta
\end{pmatrix} \in \Xi 
\end{equation*}
appear in the second coordinate of vectors in the orbit $\mathcal{A}_1 . v_0^t$. This means that the set of integers of the form
\begin{equation*}
A_0 \alpha_1^2 + A_0 \alpha_2^2 + D_0 \beta_1^2 + D_0 \beta_2^2 + 2\alpha_1 \beta_2 B_0 - 2\alpha_2 \beta_1 B_0 + 2\alpha_1 \beta_1 C_0 + 2 \alpha_2 \beta_2 C_0 - a_0 
\end{equation*}
with $\begin{pmatrix}
\alpha & \beta \\ 
\gamma & \delta
\end{pmatrix} \in \Xi$ and $\alpha=\alpha_1 + i \alpha_2$ and $\beta=\beta_1 + i \beta_2$, is a subset of the set of integers appearing in the second coordinate of vectors in the orbit $\mathcal{A}_1 . v_0^t$. From the definition of $\Xi$, this means that the set of integers of the form
\begin{equation*}
f_{a_0}(x,y,z,t) - a_0
\end{equation*}
with $x+iy$ and $z+it$ coprime in $\mathbb{Z}[i]$, $x+iy \equiv 1 \text{ or } i \bmod{2}$ and
\begin{equation*}
f_{a_0}(x,y,z,t)=A_0 x^2 + A_0 y^2 + 4D_0 z^2 + 4D_0 t^2 + 4xt B_0 - 4 yz B_0 + 4xz C_0 + 4yt C_0 
\end{equation*}
is a subset of the set of curvatures in the packing. Since, from our choice of $v_{\mathcal{P}}^t$, $A_0=a_0+b_0$ is odd, if $m$ is an odd integer which can be written $m=f_{a_0}(x,y,z,t)$ with $x+iy$ and $z+it$ coprime in $\mathbb{Z}[i]$, we automatically have that $x$ and $y$ are not of the same parity, i.e., that $x+iy \equiv 1 \text{ or } i \bmod{2}$.
\end{proof}

\begin{remark}
\begin{equation*}
\text{disc}(f_{a_0})=16\left( A_0 D_0 - 4(B_0^2 + C_0^2) \right)^2 = 16 a_0^4
\end{equation*}
and one can check that $f_{a_0}$ is positive definite (since $A_0$ is positive, all the leading principal minors are positive).
\end{remark}

\begin{remark} \label{remk}
Let $z=z_1 + iz_2 \in \mathbb{Z}[i]$. Then,
\begin{align*}
\vert z_1 + iz_2 \vert^2 f_{a_0}(x_1,x_2,x_3,x_4) &= f_{a_0}(z_1 x_1- z_2 x_2, z_2 x_1 + z_1 x_2,z_1 x_3- z_2 x_4, z_2 x_3 + z_1 x_4) \\
&=f_{a_0} \left( \Re (z(x_1 + ix_2)), \Im (z(x_1 + ix_2)), \Re (z(x_3 + ix_4)), \Im (z(x_3 + ix_4))  \right) \, .
\end{align*}
\end{remark}

\subsection{The integers $\mathbb{Z}[i]$-primitively represented by $f_{a_0}$}\

We will study the integers coprime to $disc(f_{a_0})$ which can be written as $m=f_{a_0}(x,y,z,t)=f_{a_0}(x+iy,z+it)$ with $x+iy$ and $z+it$ coprime in $\mathbb{Z}[i]$. Such a representation will be called $\mathbb{Z}[i]$-primitive. We will use known results about the representation of integers by positive quaternary quadratic forms to obtain results about $\mathbb{Z}[i]$-primitive representations.
\\

To handle the coprimality in $\mathbb{Z}[i]$, we need the M\"{o}bius function generalized to Gaussian integers. We recall that any ideal $I$ of $\mathbb{Z}[i]$ can be factored as
\begin{equation*}
I=\mathfrak{p}_1^{\alpha_1} \dots \mathfrak{p}_k^{\alpha_k}
\end{equation*}
where the $\mathfrak{p}_i$ are prime ideals of $\mathbb{Z}[i]$. This factorization is unique, up to permutations of the factors. Using this factorization, we define the function $\mu$ on the ideals of $\mathbb{Z}[i]$ as
\begin{equation*}
\mu(I)=
\begin{cases} 
      0 & \text{ if } \alpha_i \geq 2 \text{ for some } i, \\
      (-1)^k & \text{ otherwise }.
\end{cases}
\end{equation*}
\\
Analogously to the case of the integers, $\mu$ is multiplicative, and we have
\begin{equation} \label{mob}
\sum_{\substack{J \text{ ideal of } \mathbb{Z}[i] \\ J \supset I}} \mu(J) = 
\begin{cases} 
      1 & \text{ if } I= (1),  \\
      0 & \text{ otherwise}.
\end{cases}
\end{equation}

From now on, $m$ will be an integer coprime with $disc(f_{a_0})$. Let $\mathcal{N}(m)$ be the number of representations of $m$ by $f_{a_0}$ and $\mathcal{N}_P(m)$ those which are $\mathbb{Z}[i]$-primitive. Remark \ref{remk} allows us to associate representations of $f_{a_0}(x,y,z,t)=m$ with $gcd_{\mathbb{Z}[i]}(x+iy,z+it)=\pi \in \mathbb{Z}[i]$ with $\mathbb{Z}[i]$-primitive representations of $\frac{m}{\vert \pi \vert^2}$ by $f_{a_0}$, and yields
\begin{equation*}
\mathcal{N}(m)=\sum_{\substack{\pi \in \mathbb{Z}[i] \\ \vert \pi \vert^2 | m}} \mathcal{N}_P \left( \frac{m}{\vert \pi \vert^2} \right) = 4 \hspace{-0.3cm} \sum_{\substack{I \text{ ideal of } \mathbb{Z}[i] \\ \textit{N}(I) | m}} \mathcal{N}_P \left( \frac{m}{\textit{N}(I)} \right) \, .
\end{equation*}
\\
Using (\ref{mob}), we can invert this relation
\begin{equation} \label{inv}
\mathcal{N}_P(m)= \frac{1}{4} \sum_{\substack{I \text{ ideal of } \mathbb{Z}[i] \\ \textit{N}(I) | m}} \mu(I) \mathcal{N} \left( \frac{m}{\textit{N}(I)} \right) \, .
\end{equation}

The asymptotic formula for $\mathcal{N}(m)$ is known (see for example Corollary 1 of \cite{HB} or Theorem 20.9 of \cite{IK}).
\begin{theorem}
\begin{equation*}
\mathcal{N}(m)=\frac{\pi^2}{2a_0^2} m \mathfrak{S}(m) + O(m^{\frac{3}{4}+\varepsilon})
\end{equation*}
for every $\varepsilon>0$, where
\begin{equation*}
\mathfrak{S}(m)=\prod_{p} \delta_p(n)
\end{equation*}
with
\begin{equation*}
\delta_p(n)=\lim_{k \rightarrow \infty} p^{-3k} \vert x \in \left( \mathbb{Z}/p^k \mathbb{Z} \right)^4 \text{ such that } f_{a_0}(x) \equiv m \bmod{p^k} \vert \, .
\end{equation*}
The implied constant depends only on $\varepsilon$.
\end{theorem}

From this and (\ref{inv}), we obtain
\begin{equation} \label{equaref}
\mathcal{N}_P(m)=\frac{\pi^2}{8a_0^2} m \sum_{\substack{I \text{ ideal of } \mathbb{Z}[i] \\ \textit{N}(I) | m}} \frac{\mu(I)}{\textit{N}(I)} \mathfrak{S} \left( \frac{m}{\textit{N}(I)} \right) + O \left( m^{\frac{3}{4}+\varepsilon} \sum_{\substack{I \text{ ideal of } \mathbb{Z}[i] \\ \textit{N}(I) | m}} \frac{1}{\textit{N}(I)^{\frac{3}{4}+\varepsilon}} \right) \, .
\end{equation}

\subsubsection{The error term}\

\begin{lemma} \label{lemmaerror}
The error term in (\ref{equaref}) is $O(m^{\frac{3}{4}+\varepsilon})$ for any $\varepsilon > 0$ (and the implied constant depends only on $\varepsilon$).
\end{lemma}

\begin{proof}
Let $r_2(d)$ be the number of representations of $d$ as a sum of two squares. Since $r_2(d) \ll \tau(d)$,
\begin{align*}
\sum_{\substack{I \text{ ideal of } \mathbb{Z}[i] \\ \textit{N}(I) | m}} \frac{1}{\textit{N}(I)^{\frac{3}{4}+\varepsilon}} &= \frac{1}{4} \sum_{d|m} \frac{r_2(d)}{d^{\frac{3}{4}+\varepsilon}} \\
& \ll \sum_{d|m} \frac{\tau(d)}{d^{\frac{3}{4}+\varepsilon}} \\
& \ll \tau(m) \sum_{d|m} \frac{1}{d^{\frac{3}{4}+\varepsilon}} \\
& \ll \tau(m)^2 \, .
\end{align*}
Therefore, the error term is $O(m^{\frac{3}{4}+\varepsilon})$ for any $\varepsilon > 0$.
\end{proof}
\subsubsection{The main term}\

To get a more explicit expression for the main term, we need a better understanding of the local densities $\delta_p (\frac{m}{\textit{N}(I)})$. This can be achieved using the following lemma (see, for example, \cite{sie}).
\begin{lemma} \label{sieg}
Let $p \not | disc(f_{a_0})$, then
\begin{equation*}
\delta_p(m)=\left( 1-\frac{1}{p^2} \right) \left( 1+\frac{1}{p} \dots + \frac{1}{p^{v_p(m)}} \right)
\end{equation*}
where $v_p(m)$ is such that $p^{v_p(m)} || m$.
\end{lemma}

\begin{lemma} \label{lemprim}
Let $(p,m)=1$. Then, $\delta_p(\frac{m}{\textit{N}(I)})=\delta_p(m)$ for any ideal $I$ of $\mathbb{Z}[i]$ with $\textit{N}(I) | m$.
\end{lemma}

\begin{proof}
\begin{equation*}
\delta_p \left( \frac{m}{\textit{N}(I)} \right) = \lim_{k \rightarrow \infty} p^{-3k} A_{p^k} \left( \frac{m}{\textit{N}(I)} \right)
\end{equation*}
with
\begin{align*}
A_{p^k} \left( \frac{m}{\textit{N}(I)} \right) &=\vert x \in \left( \mathbb{Z}/p^k \mathbb{Z} \right)^4 \text{ such that } f_{a_0}(x) = \frac{m}{\textit{N}(I)} \bmod{p^k} \vert \\
&=\frac{1}{p^k} \sum_{x \in \left( \mathbb{Z}/p^k \mathbb{Z} \right)^4} \sum_{h \in \mathbb{Z}/p^k \mathbb{Z}} \e \left( \frac{(f_{a_0}(x) - \frac{m}{\textit{N}(I)})h}{p^k} \right) \, .
\end{align*}
Using the change of variables $h \longrightarrow \textit{N}(I)h$ (which is an isomorphism of $\mathbb{Z}/p^k \mathbb{Z}$ because $\textit{N}(I) | m$ and $(p,m)=1$) 
\begin{equation*}
A_{p^k} \left( \frac{m}{\textit{N}(I)} \right) = \frac{1}{p^k} \sum_{x \in \left( \mathbb{Z}/p^k \mathbb{Z} \right)^4} \sum_{h \in \mathbb{Z}/p^k \mathbb{Z}} \e \left( \frac{(\textit{N}(I)f(x) - m)h}{p^k} \right) \, .
\end{equation*}
$I=(z_1 + i z_2)$ for some $z_1 + i z_2 \in \mathbb{Z}[i]$ (since $\mathbb{Z}[i]$ is a principal ideal domain), therefore, using Remark \ref{remk},
\begin{equation*}
\textit{N}(I)f_{a_0}(x)=f_{a_0}(xA)
\end{equation*}
with
\begin{equation*}
A=\begin{pmatrix}
z_1 & z_2 & 0 & 0 \\ 
-z_2 & z_1 & 0 & 0 \\ 
0 & 0 & z_1 & z_2 \\ 
0 & 0 & -z_2 & z_1
\end{pmatrix} 
\end{equation*}
and the map $x \longrightarrow xA$ is an isomorphism of $\left( \mathbb{Z}/p^k \mathbb{Z} \right)^4$ (because $p \not \vert N(I) = det(A)$). Thus
\begin{align*}
A_{p^k} \left( \frac{m}{\textit{N}(I)} \right) &= \frac{1}{p^k} \sum_{x \in \left( \mathbb{Z}/p^k \mathbb{Z} \right)^4} \sum_{h \in \mathbb{Z}/p^k \mathbb{Z}} \e \left( \frac{(f_{a_0}(x) - m)h}{p^k} \right) \\
&= A_{p^k}(m) \, .
\end{align*}
\end{proof}

\begin{lemma} \label{lemmamain}
The main term in (\ref{equaref}) is
\begin{equation*}
\frac{\pi^2}{8a_0^2} m \mathfrak{S}(m) \prod_{\substack{p \equiv 1 \bmod{4} \\ p | m}} \left( 1-\frac{1}{p} \right)^2 \left( 1- \frac{1}{p^{v_p(m)+1}} \right)^{-2} \prod_{\substack{p \equiv 3 \bmod{4} \\ p^2 | m}} \left( 1-\frac{1}{p^2} \right) \left( 1- \frac{1}{p^{v_p(m)+1}} \right)^{-1} \, .
\end{equation*}
\end{lemma}

\begin{proof}
Using Lemma \ref{sieg} and Lemma \ref{lemprim},
\begin{align*}
\mathfrak{S} \left( \frac{m}{\textit{N}(I)} \right) &= \prod_{(p,m)=1} \delta_p \left( \frac{m}{\textit{N}(I)} \right) \prod_{p|m} \delta_p \left( \frac{m}{\textit{N}(I)} \right) \\
&=  \prod_{(p,m)=1} \delta_p(m) \prod_{p|m} \left( 1-\frac{1}{p^2} \right) \left( 1+\frac{1}{p} \dots + \frac{1}{p^{v_p(\frac{m}{\textit{N}(I)})}} \right) \\
&= \mathfrak{S}(m) \prod_{p| \textit{N}(I)} \left( \frac{1+\frac{1}{p} \dots + \frac{1}{p^{v_p(\frac{m}{\textit{N}(I)})}}}{1+\frac{1}{p} \dots + \frac{1}{p^{v_p(m)}}} \right) \, .
\end{align*}
Therefore, the main term is given by
\begin{equation*}
\frac{\pi^2}{8a_0^2} m\mathfrak{S}(m) \hspace{-0.3cm} \sum_{\substack{I \text{ ideal of } \mathbb{Z}[i] \\ \textit{N}(I) | m}} g(I)
\end{equation*}
with
\begin{equation*}
g(I)=\frac{\mu(I)}{\textit{N}(I)} \prod_{p| \textit{N}(I)} \left( \frac{1+\frac{1}{p} \dots + \frac{1}{p^{v_p(\frac{m}{\textit{N}(I)})}}}{1+\frac{1}{p} \dots + \frac{1}{p^{v_p(m)}}} \right) \, .
\end{equation*}

In $\mathbb{Z}[i]$, there are only two possible types of norms for a prime ideal of norm $\ne 2$. Let $\mathfrak{p}$ be an ideal with $\textit{N}(\mathfrak{p}) \ne 2$, then
\begin{equation*}
\textit{N}(\mathfrak{p}) = p \text{ with } p \equiv 1 \bmod{4} \quad \text{ or } \quad \textit{N}(\mathfrak{p}) = p^2 \text{ with } p \equiv 3 \bmod{4} \, .
\end{equation*}
Reciprocally, for any $p \ne 2$, there exists exactly two ideals of norm $p$ if $p \equiv 1 \bmod{4}$ and a unique ideal of norm $p^2$ if $p \equiv 3 \bmod{4}$, and these ideals are prime.
\\
\\
Notice that $g$ is multiplicative, and that, if $\mathfrak{p}$ is a prime ideal of $\mathbb{Z}[i]$ with $\textit{N}(\mathfrak{p})|m$, we have
\begin{equation*}
g(\mathfrak{p})=\left( 1+\frac{1}{p} \dots + \frac{1}{p^{v_p(m)-1}} \right) \left( p(1+\frac{1}{p} \dots + \frac{1}{p^{v_p(m)}}) \right)^{-1} \text{ if } \textit{N}(\mathfrak{p})=p \text{ with } p \equiv 1 \bmod{4} \, ,
\end{equation*}
\begin{equation*}
g(\mathfrak{p})=\left( 1+\frac{1}{p} \dots + \frac{1}{p^{v_p(m)-2}} \right) \left( p^2(1+\frac{1}{p} \dots + \frac{1}{p^{v_p(m)}}) \right)^{-1} \text{ if } \textit{N}(\mathfrak{p})=p^2 \text{ with } p \equiv 3 \bmod{4} \, .
\end{equation*}
Therefore, the sum can be transformed into a product, and we get
\begin{align*}
\sum_{\substack{I \text{ ideal of } \mathbb{Z}[i] \\ \textit{N}(I) | m}} g(I) &= \prod_{\substack{p \equiv 1 \bmod{4} \\ p | m}} \left( 1 - \frac{1+\frac{1}{p} \dots + \frac{1}{p^{v_p(m)-1}}}{p(1+\frac{1}{p} \dots + \frac{1}{p^{v_p(m)}})} \right)^2 \prod_{\substack{p \equiv 3 \bmod{4} \\ p^2 | m}} \left( 1 - \frac{1+\frac{1}{p} \dots + \frac{1}{p^{v_p(m)-2}}}{p^2(1+\frac{1}{p} \dots + \frac{1}{p^{v_p(m)}})} \right) \\
&= \prod_{\substack{p \equiv 1 \bmod{4} \\ p | m}} \left( 1-\frac{1}{p} \right)^2 \left( 1- \frac{1}{p^{v_p(m)+1}} \right) ^{-2} \prod_{\substack{p \equiv 3 \bmod{4} \\ p^2 | m}} \left( 1-\frac{1}{p^2} \right) \left( 1- \frac{1}{p^{v_p(m)+1}} \right)^{-1} \, .
\end{align*}
\end{proof}

\begin{lemma} \label{pospos}
\begin{equation*}
\mathcal{N}_P(m)=\frac{\pi^2}{8a_0^2} m \mathfrak{S}(m) \hspace{-0.2cm} \prod_{\substack{p \equiv 1 \bmod{4} \\ p | m}} \hspace{-0.1cm} \left( 1-\frac{1}{p} \right)^2 \left( 1- \frac{1}{p^{v_p(m)+1}} \right) ^{-2} \hspace{-0.3cm} \prod_{\substack{p \equiv 3 \bmod{4} \\ p^2 | m}} \hspace{-0.1cm} \left( 1-\frac{1}{p^2} \right) \left( 1- \frac{1}{p^{v_p(m)+1}} \right)^{-1} \hspace{-0.2cm} + O(m^{\frac{3}{4}+\varepsilon}) \, .
\end{equation*}
\end{lemma}

\begin{proof}
This is a direct consequence of Lemmas \ref{lemmaerror} and \ref{lemmamain}.
\end{proof}

\subsection{The local-global principle}\

\begin{theorem} \label{positi}
Let $m$ be an integer coprime with $disc(f_{a_0})$. Suppose that the equation $f_{a_0}(x) \equiv m \bmod{p}$ has solutions for every $p|disc(f_{a_0})$, $p \ne 2$, and that the equation $f_{a_0}(x) \equiv m \bmod{8}$ has solutions. Then, for such an $m$, with $m$ large enough, $m$ is $\mathbb{Z}[i]$-primitively represented by $f_{a_0}$.
\end{theorem}

\begin{proof}
We use Lemma \ref{pospos}.
\\
\\
Using Mertens' formula,
\begin{equation*}
\prod_{\substack{p \equiv 1 \bmod{4} \\ p | m}} \left( 1-\frac{1}{p} \right) \geq \prod_{p \leq m} \left( 1-\frac{1}{p} \right) \sim \frac{\e^{-\gamma}}{\log m} \, .
\end{equation*}
Furthermore,
\begin{equation*}
\prod_{\substack{p \equiv 3 \bmod{4} \\ p^2 | m}} \left( 1-\frac{1}{p^2} \right) \geq \prod_{p} \left( 1-\frac{1}{p^2} \right) =\frac{1}{\zeta(2)}
\end{equation*}
and
\begin{equation*}
\prod_{\substack{p \equiv 1 \bmod{4} \\ p | m}} \left( 1- \frac{1}{p^{v_p(m)+1}} \right) ^{-1} \prod_{\substack{p \equiv 3 \bmod{4} \\ p^2 | m}} \left( 1- \frac{1}{p^{v_p(m)+1}} \right)^{-1} \geq 1 \, .
\end{equation*}

We now have to show that $\mathfrak{S}(m)$ is bounded away from zero. Using Lemma \ref{sieg},
\begin{equation*}
\prod_{p \not \vert disc(f_{a_0})} \delta_p(m) \gg 1 \, .
\end{equation*}
It remains to prove that $\delta_p(m)$ is bounded away from zero for the finite set of primes $p \vert disc(f_{a_0})$.
\\
\\
Using Lemma 13 of \cite{sie}, for $p \ne 2$, $\delta_p(m) = p^{-3}\vert x \in (\mathbb{Z}/p\mathbb{Z})^4 \text{ such that } f_{a_0}(x) \equiv m \bmod{p} \vert$.
\\
But, by assumption, there exist solutions to $f_{a_0}(x) \equiv m \bmod{p}$, therefore $\delta_p(m)>0$ for those $p$.
\\
\\
Finally, using Lemma 13 of \cite{sie} again, for $p=2$, $\delta_2(m) = \frac{1}{512}\vert x \in (\mathbb{Z}/8\mathbb{Z})^4 \text{ such that } f_{a_0}(x) \equiv m \bmod{8} \vert$. As before, this is positive.
\end{proof}

\begin{theorem}[The local-global principle]
Let 
\begin{equation*}
\mathcal{S} = \left\lbrace n \in \mathbb{Z}, n > 0 \text{ such that } n \equiv b_0 \bmod{4} \right\rbrace \, .
\end{equation*}
Then, an integer $m$ large enough with $gcd(m,a_0)=1$ is the curvature of some sphere in the packing if and only if $m \in \mathcal{S}$.
\end{theorem}

\begin{proof}
We will look for the integers $m$ satisfying the following conditions:
\begin{itemize}
\item $m+a_0$ is odd,
\item $gcd(m+a_0,disc(f_{a_0}))=1$,
\item $m+a_0$ is represented by $f_{a_0}$ modulo $8$ and modulo any odd prime dividing the discriminant.
\end{itemize}
From our choice of $v_{\mathcal{P}}^t$, $a_0+b_0$ is odd. As seen before, this implies that $f_{a_0}$ will only $\mathbb{Z}[i]$-primitively represent odd values. That is why we need $m+a_0$ to be odd.
\\
\\
We also want $gcd(m+a_0,disc(f_{a_0}))=gcd(m+a_0,16a_0^4)=1$ to apply Theorem \ref{positi}. But since $m+a_0$ is odd, it simply means $gcd(m,a_0)=1$.
\\
\\
We want $m+a_0$ to be represented by $f_{a_0}$ modulo $8$. But, since $A_0$ is odd, it is easy to check that $f_{a_0}$ represents exactly the two odd classes modulo $8$, $A_0$ and $5A_0 \bmod{8}$. We can therefore just look modulo $4$, and $m+a_0$ needs to be in the class $A_0 \bmod{4}$, i.e., $m \equiv b_0 \bmod{4}$. Notice that this condition contains the condition $m+a_0$ odd.
\\
\\
We want $m+a_0$ to be represented by $f_{a_0}$ modulo any odd prime dividing the discriminant. Notice that, since $A_0$ is odd, 
\begin{align*}
gcd(A_0,B_0,C_0,D_0) &= gcd(a_0+b_0,a_0+b_0+c_0+d_0-2\omega_0,a_0+b_0+c_0-d_0,a_0+c_0) \\
&=gcd(a_0+b_0,a_0+c_0,a_0+d_0,a_0+\omega_0) \, .
\end{align*}
But, from equation (\ref{eqmod}),
\begin{align*}
p\ne 2 \text{ and } p \vert gcd(a_0+b_0,a_0+c_0,a_0+d_0,a_0+\omega_0) &\Rightarrow p|a_0 \\
&\Rightarrow p \vert gcd(a_0,b_0,c_0,d_0,\omega_0)
\end{align*}
which is a contradiction to the primitivity. Now, one can easily check that this means that $f_{a_0}$ can take all the possible values modulo $p$. Hence, this imposes no restriction on $m+a_0$, except the one already seen before $gcd(m,a_0)=1$.
\\
\\
From Theorem \ref{curva}, we know that the set of integers of the form
\begin{equation*}
f_{a_0}(x,y,z,t)-a_0 \quad gcd_{\mathbb{Z}[i]}(x+iy,z+it)=1
\end{equation*}
is a subset of the set of curvatures in $\mathcal{P}$. From Theorem \ref{positi} and the previous observations, we deduce that all the integers $m$ large enough with $gcd(m,a_0)=1$ and $m \equiv b_0 \bmod{4}$ are $\mathbb{Z}[i]$-primitively represented by $f_{a_0}(x,y,z,t)-a_0$, and therefore are curvatures of some sphere in the packing.
\\
\\
Reciprocally, from Lemma $\ref{odd}$, any curvature $m$ in the packing with $gcd(m,a_0)=1$ will verify $m \equiv b_0 \bmod{4}$.
\end{proof}

\begin{remark}
This result gives us that the set of curvatures in an integral generalized Apollonian sphere packing has positive density. More precisely, the density is at least
\begin{equation*}
\frac{1}{4} \prod_{\substack{p \vert n \\ p \ne 2}} \left( 1-\frac{1}{p} \right) \, .
\end{equation*}
\end{remark}

\bigskip

\nocite{*}
\bibliographystyle{alpha}
\bibliography{genspheres}

\textsc{\footnotesize D\'{e}partement de math\'{e}matiques et statistiques,
Universit\'{e} de Montr\'{e}al,
CP 6128 succ. Centre-Ville,
Montr\'{e}al QC H3C 3J7, Canada
}

\textit{\small Email address: }\texttt{\small dimitrid@dms.umontreal.ca}

\end{document}